\documentclass[a4paper,11pt]{article}
\usepackage[utf8]{inputenc}
\usepackage{amsmath, amsfonts, amsthm,amssymb,  bbm,hyphenat,booktabs,graphicx,
setspace,mathrsfs}
\usepackage{geometry}
\usepackage{math}
\usepackage{hyperref} 
\usepackage[usenames, dvipsnames]{color}
\newtheorem{lemma}{Lemma}[section]
\newtheorem{theorem}[lemma]{Theorem}

\newtheorem{remark}[lemma]{Remark}
\newtheorem{example}[lemma]{Example}

\newtheorem{definition}[lemma]{Definition}

\newcommand{\di}{{\rm d}}
\newcommand{\hos}{{\rm ho}}
\newcommand{\tor}{{\rm to}}
\newcommand{\dom}{{\rm Dom}}

\newcommand{\vertiii}[1]{{\left\vert\kern-0.25ex\left\vert\kern-0.25ex\left\vert #1 
    \right\vert\kern-0.25ex\right\vert\kern-0.25ex\right\vert}}

\title{Lower bounds on the growth of  Sobolev norms in some linear time dependent Schr\"odinger equations}

\author{
A. Maspero\footnote{ International School for Advanced Studies (SISSA), Via Bonomea 265, 34136, Trieste, Italy \newline
 \textit{Email: } \texttt{alberto.maspero@sissa.it}}
}

\begin{document} 
\maketitle{} 

\begin{abstract}
In this paper we consider linear, time dependent Schr\"odinger equations of the form 
$\im \partial_t  \psi = K_0 \psi + V(t) \psi $, where $K_0$ is a positive self-adjoint operator with discrete spectrum and
whose  spectral gaps are asymptotically constant. \\
 We  give a strategy to construct bounded perturbations $V(t)$ such that the Hamiltonian $K_0 + V(t)$ generates   unbounded orbits.
We apply our abstract construction to three cases: (i) the  Harmonic oscillator on $\R$, (ii) the half-wave equation on $\T$ and (iii) the Dirac-Schr\"odinger equation on Zoll manifolds.
In each case, $V(t)$ is a  smooth and periodic in time pseudodifferential operator and the Schr\"odinger equation  has solutions fulfilling 
$\norm{\psi(t)}_r \gtrsim |t|^{r }$ as $|t| \gg 1$.
\end{abstract} 

\section{Introduction}
In this paper we study linear Schr\"odinger equations of the form 
\begin{equation}
\label{LS}
\im \partial_t  \psi = K_0 \psi + V(t) \psi 
\end{equation}
on a scale of Hilbert spaces $\cH^r$. 
Here   $K_0$ is a positive,  selfadjoint operator with purely discrete spectrum,  $V(t)$ is a time dependent self-adjoint perturbation, and the scale  $\cH^r$  is the one  defined spectrally by  $K_0$.

We develop an abstract technique to construct  smooth and periodic in time, bounded operators $V(t)$ for which  \eqref{LS} has unbounded orbits with norms growing polynomially in time, 
despite   every  orbit of the  unperturbed flow  being   bounded. 
%
%
%

Our main result (Theorem \ref{thm:main}) is to develop a  general procedure to construct such perturbations 
{\em in case  $K_0$  has constant spectral gaps}.\\
 We successfully apply this strategy   to three different models: (i) the  Harmonic oscillator on $\R$, (ii) the half-wave equation on $\T$ and (iii) the Dirac-Schr\"odinger equation on Zoll manifolds.
In each case, we construct $V(t)$ as a pseudodifferential operator of order 0,  smooth and $2\pi$-periodic  in time,  so that the Hamiltonian  $K_0 + V(t)$ has solutions $\psi(t)$ fulfilling 
\begin{equation}
\label{gr}
\norm{\psi(t)}_r \geq C_r \, t^{ r}  \ , \qquad \mbox{ for } t \gg 1 \ ,
\end{equation}
which display,  therefore, growth of Sobolev norms.\\
Note that case (iii) differs from (i) and (ii), since  the Dirac-Schr\"odinger operator on Zoll manifolds has only {\em asymptotically constant spectral gaps};   however, such  operator is  a smoothing perturbation of  a $K_0$   with  constant spectral gaps, and our method applies with just a minor modification. 
In particular, the difference between cases (i)-(ii) and (iii) is that in the former ones the perturbations $V(t)$ can be arbitrary small in size; on the contrary, in case (iii)  $V(t)$ has to  contain a not perturbative term to correct the  spectral gaps.

The problem of constructing unbounded solutions in Schr\"odinger equations has recently attracted a lot of attention. However, even  in 
the simpler case of {\em linear time dependent } equations there are not   many results in the literature.
Up to our knowledge, the only  examples 
were given by Bourgain for a Klein-Gordon and Schr\"odinger equation on $\T$ \cite{bou99}, by Delort for the Harmonic oscillator on $\R$ \cite{del}, and by Bambusi,  Gr\'ebert, Robert and the author for the Harmonic oscillators on $\R^d$, $d \geq 1$ \cite{BGMR1}.
While the techniques of \cite{bou99, del} are quite involved, 
the construction of  \cite{BGMR1} is simpler and based on 
a result by Graffi and Yajima \cite{GY00} 
to prove  stability of the absolutely  continuous spectrum 
of a certain Floquet operator.\\
It turns out that  the mechanism of Graffi and Yajima is quite general, and the main idea of this paper is that their  procedure can be  ``reversed", in an abstract setting, 
to create perturbations which provoke growth of norms.
%
%
It is therefore  worth to spend few words on the  result by  Graffi and Yajima. \cite{GY00} considers the following periodically forced Harmonic oscillator 
\begin{equation}
\label{eq:gy}
\im \dot \psi = \frac{1}{2}\left(-\partial_{xx} + x^2\right) \psi + a\,  x(\sin t)\psi  \ , \qquad x \in \R \ , \ a \in \R\setminus\{ 0 \}  \ , 
\end{equation} as an example of
Hamiltonian whose Floquet spectrum is absolutely continuous, despite the fact that the unperturbed
Hamiltonian  has discrete spectrum. 
This statement is proved  by conjugating \eqref{eq:gy} 
by the flow of the unperturbed Hamiltonian and then by a  Galilean transform; 
the conjugated Hamiltonian is simply $\frac{a}{2\im}\partial_x$, 
whose Floquet spectrum is  absolutely continuous. 
In  \cite{BGMR1}, this scheme  is  exploited to prove growth of  norms  for \eqref{eq:gy}, by simply remarking that 
 $\frac{a}{2\im}\partial_x$ generates orbits unbounded w.r.t. the norms defined spectrally by the harmonic oscillator (actually  \cite{BGMR1} considers also perturbations quasi-periodic in time).

Now, the main idea of this paper is that it is possible to ``reverse'' the mechanism of \cite{GY00} to construct time dependent equations with unbounded orbits.
The construction is the following: assume to have  a {\em time independent} pseudodifferential operator $A$ whose  propagator generates unbounded orbits, and define the time dependent perturbation 
\begin{equation}
\label{def:VA}
V_A(t) := e^{-\im t K_0 } \, A \, e^{\im t K_0}  
\end{equation}
(i.e. conjugate $A$ by the inverse flow of the unperturbed operator).
By construction, the dynamics of $K_0 + V_A(t)$
is unitarily equivalent to the dynamics of $A$, so
it has unbounded orbits as well. Provided the map $t \mapsto V_A(t)$ is smooth and periodic in time, we have found our perturbation; it is here that we use (indirectly) that $K_0$ has constant spectral gaps, see Remark \ref{rem:K0.csg}.

Therefore, all is left to do in applications is to find a  time independent  operator whose  propagator generates  unbounded orbits.
This is a much simpler task, 
and  the general philosophy is to look for  operators with nontrivial absolutely continuous spectrum; indeed,  for these operators,  the   Guarneri-Combes theorem \cite{guarneri1, combes} guarantees the existence of initial data $\psi$ for which the time-averaged Sobolev norms   $\frac{1}{T}\int_0^T\norm{e^{-\im t A}\psi}_r \, \di t$ grow  in $T$. 
This is a slightly weaker statement than \eqref{gr}; however,  in  applications, one can typically prove the   stronger estimate    $\norm{e^{-\im t A}\psi}_r \geq C_{r,\psi} \,  t^r$ as $t \to \infty$.
In particular we succeed in doing this for the Harmonic oscillator on $\R$, giving an alternative, shorter proof of Delort's result  \cite{del} (see Theorem \ref{thm:hos}).\\

As a further comment, it is interesting to compare the rate of growth \eqref{gr} with the upper bounds proved in \cite{maro, BGMR2} for equations of the form \eqref{LS}.
In particular, the results of \cite{maro} imply that  for any  $V(t)$
continuous in time (but otherwise arbitrary depending)   and  pseudodifferential of order $\rho \leq 1$, each solution of \eqref{LS} fulfills
\begin{equation}
\label{ub1}
\norm{\psi(t)}_r \leq C_{r, \rho} \la t \ra^{\frac{r}{1 - \rho} }\ , \qquad \forall t \in \R \ . 
\end{equation}
Therefore, the solutions  we construct here saturate   \eqref{ub1}, at least for   $\rho =0$. 

The upper bound \eqref{ub1} can be improved  adding the assumption that  $V(t)\equiv V(\omega t)$ is quasiperiodic in time with a nonresonant frequency vector $\omega \in \R^n$. Indeed, in \cite{BGMR2} it is proved that  if $\omega$ fulfills the nonresonance condition 
\begin{equation}
\label{nonres}
\exists \gamma, \tau >0 \ \ \mbox{s.t. } \ \ \abs{  \ell + \omega \cdot k } \geq \frac{\gamma}{(1+ |k| + |\ell|)^\tau} \ , \qquad \forall \ell \in \Z\setminus\{0\} , \ \forall k \in \Z^n\setminus\{0\} \ , 
\end{equation}
then for any $r \geq 0$, any $\epsilon >0$ arbitrary small,  \eqref{ub1} improves to 
\begin{equation}
\label{ub2}
\norm{ \psi(t)}_r \leq  C_{r,\epsilon} \la t \ra^\epsilon \ , \qquad \forall t \in \R \ ;
\end{equation}
the last estimate means that the growth of norms, if happens, is subpolynomial in time.
Note that \eqref{ub2}  is not in contrast with the faster  growth of the norms \eqref{gr} ; indeed  the spectral condition that we impose on $K_0$ 
(see \eqref{AssA}) implies that $V_A(t)$ defined in \eqref{def:VA} is periodic in time with frequency $\omega = 1$, 
which is  clearly resonant.

Finally, in some cases one can prove that the Sobolev norms 
of the solution stay uniformly bounded in time. 
This requires typically 
nonresonance conditions stronger than  \eqref{nonres} and a smallness assumption on the size of the perturbation. 
In this case,   one might try  to prove a reducibility KAM theorem, conjugating
$K_0 + V(t)$
 to a new Hamiltonian which is time independent and commutes  with $K_0$; as a consequence, one gets the upper bound 
\begin{equation}
\label{up3}
\sup_{t \in \R} \norm{\psi(t)}_r \leq C_r \ .
\end{equation}
Concerning the systems that we treat here, the reducibility scheme has been successfully 
implemented for the Harmonic oscillators on $\R^d$ \cite{C87, W08, GT11, GP,  Bam16I, BGMR1},   wave equations on the torus  \cite{CY00, fang14, mont17} (these methods can be used to prove  reducibility for the half-wave equation on $\T$), 
and Klein-Gordon equation on the sphere \cite{GP2}.
In all cases the frequency $\omega$ must be chosen in a Cantor set of nonresonant vectors and the perturbation must be sufficiently small in size. 
We recall that  also the perturbations constructed here (and which provoke growth of norms) can be arbitrary small in size; 
therefore the stability/instability of the system depends only on the resonance property of the frequency $\omega$.

Before closing this introduction, let us mention that, in case of {\em nonlinear} Schr\"odinger equations,  the problem of constructing solutions with unbounded orbits  is extremely difficult. 
A first breakthrough was achieved in \cite{CKSTT}, which constructs solutions of the cubic nonlinear Schr\"odinger equation on $\T^2$ whose Sobolev norms become arbitrary large 
(see also \cite{hani14, guardia14, haus_procesi15, guardia_haus_procesi16}
 for generalizations of this result);
  however  the existence of unbounded orbits for this model is still an open problem.\\
At the moment, existence of unbounded  orbits  has only been proved by G\'erard and Grellier \cite{gerard_grellier} for the cubic 
Szeg\H o equation on $\T$, and by 
Hani, Pausader, Tzvetkov and 
              Visciglia \cite{hani15} for the cubic NLS  on $\R \times \T^2$.

\vspace{1em}
\noindent{\bf Acknowledgments.}
We wish to thank Dario Bambusi, Didier Robert and Beno\^it Gr\'ebert  for many useful discussions and suggestions during the preparation of the manuscript.
Currently we are partially supported by PRIN 2015 ``Variational methods, with applications to problems in mathematical physics and geometry".

\section{An abstract theorem of growth}
\label{sec:alg}
We start with a Hilbert space $(\cH, \la \cdot , \cdot \ra)$ and a reference operator $K_0$,
which we assume to be selfadjoint and positive, namely such that
$$
\langle \psi; K_0\psi\rangle\geq c_{K_0} \norm{\psi}_0^2\ ,\quad \forall
\psi\in \dom(K_0^{1/2})\ ,\quad c_{K_0} >0\ , 
$$ and define as usual a scale of Hilbert spaces by $\cH^r=\dom(K_0^r)$
(the domain of the operator $K_0^r$) if $r\geq 0$, and
$\cH^{r}=(\cH^{-r})^\prime$ (the dual space) if $ r<0$.
 We endow $\cH^r$ with the norm
$$
\norm{\psi}_r:= \norm{K_0^r \psi}_{0} \ , 
$$ 
where $\norm{\cdot}_0$ is
the norm of $\cH^0 \equiv \cH$.  Finally we
denote by $\cH^{-\infty} = \bigcup_{r\in\R}\cH^r$ and $\cH^{+\infty} =
\bigcap_{r\in\R}\cH^r$.   Notice that 
$\cH^{+\infty}$ is a dense linear subspace of $\cH^r$ for any $r\in\R$ (this is a
consequence of the spectral decomposition of $K_0$).

\begin{remark}
By the very definition of $\cH^r$, the unperturbed flow $e^{-\im t K_0}$ preserves each norm, $\norm{e^{-\im t K_0} \psi }_r = \norm{\psi}_r$ $\, \forall t \in \R$. 
Consequently, every orbit of the unperturbed equation is bounded.
\end{remark}

For $\cX, \cY$ Banach spaces,  we denote by $\cL(\cX, \cY)$ the set of bounded operators from $\cX$ to $\cY$; if $\cX =\cY$, we simply write  $\cL(\cX)$.

We state now the assumptions that we need in order to prove our result. The first one  is a spectral property of $K_0$:\\

 \noindent{{\bf Assumption A:}}  $K_0$ has an entire discrete spectrum such that  
 \begin{equation}
 \label{AssA}
 {\rm spec}(K_0) \subseteq \N +\lambda
 \end{equation}   for some  $\lambda \geq 0$.\\

\begin{remark}
\label{rem:flow.per}
Assumption {\rm A} guarantees that
$e^{\im 2\pi K_0} = e^{\im 2 \pi \lambda}$.
In particular, for any operator $V$, the map  $t \mapsto e^{\im t K_0} V e^{-\im t K_0}$ is $2\pi$-periodic.
\end{remark}
The next assumption regards the existence of a bounded, {\em  time independent} operator with unbounded orbits:\\

\noindent{{\bf Assumption B:}} There exists an operator $A \in \cL(\cH^r)$,  an initial datum $\psi_0 \in \cH^r$ and a real $\mu >0$ such that  the  Schr\"odinger equation
\begin{equation}
\label{eq:A}
\im \dot \psi = A \psi \ , \qquad \ \psi(0) = \psi_0
\end{equation}
 has a solution $\psi(t) \in \cH^r$ fulfilling
 \begin{equation}
 \label{growth.os}
 \norm{\psi(t)}_r \geq C_r \, t^{\mu r} \ , \qquad t \gg 1 \ .
 \end{equation}

Despite  Assumption B might seem very strong,  in applications it can be verified quite easily. 
In Lemma 	\ref{lem:ad.0} below we  give some sufficient conditions on the operator $A$ to obtain  \eqref{growth.os}.

The last assumption concerns smoothness in time of the map
$t\mapsto {\rm e}^{-\im t K_0}\, A \, {\rm e}^{\im t K_0}$:\\
%

\noindent{{\bf Assumption C:}} The map 
         $t\mapsto {\rm e}^{-\im t K_0}\, A \, {\rm e}^{\im t K_0}\in
         C^{\infty}\left(\R ; \cL(\cH^r)\right)$.\\
 
 Here     $C^{\infty}\left(\R ; \cL(\cH^r)\right)$ is the class of smooth maps from $\R$ to $\cL(\cH^r)$.

\begin{remark}
\label{rem:flow2}
By  Remark \ref{rem:flow.per},  $t\mapsto {\rm e}^{-\im t K_0}\, A \, {\rm e}^{\im t K_0}\in
         C^{\infty}\left(\T ; \cL(\cH^r)\right)$, namely the map and all its derivatives are  $2\pi$-periodic.
\end{remark} 
         
In  applications,  Assumption {\rm C} can be verified by requiring  $A$ and $ K_0$ to be   pseudodifferential operators and $K_0$ to fulfill an Egorov-like theorem. 

\begin{remark}
\label{rem:K0.csg}
We do not require explicitly $K_0$ to have constant spectral gaps; however, in applications, the only operators that we could find that verify both Assumption {\rm A} and {\rm C} have constant spectral gaps.
\end{remark}

Our main result is the following one:
\begin{theorem}
\label{thm:main}
Assume  {\rm A, B, C}. There exists $V_A(t) \in C^\infty(\T, \cL(\cH^r))$ s.t.   $K_0 + V_A(t)$ generates  unbounded orbits. More precisely there exists a smooth solution $\psi(t)$ of  \eqref{LS},   with $\psi(t) \in \cH^r$ $\, \forall t$, fulfilling  \eqref{growth.os}.
\end{theorem}

\begin{proof}
The proof is trivial.
Define  $V_A(t)$ as in \eqref{def:VA}. By {\rm Assumption A} and {\rm C}, it belongs to $C^\infty(\T, \cL(\cH^r))$.
The  change of coordinates  $\psi = e^{- \im t  K_0}\vf$ conjugates $\im \dot \psi = (K_0 + V_A(t))\psi$ to \eqref{eq:A}  and preserves the norm $\norm{\cdot}_r$ for any time $t$.
Then {\rm Assumption B} implies the claim.
\end{proof}

We comment now on  Assumption B. 
First remark that a necessary condition for \eqref{growth.os} to be fulfilled is that $[K_0, A] \neq 0$. 
Then a sufficient condition is that only a finite number of  iterated commutators of $A$ and $K_0$ are  not zero.
We define $\ad_A (B) := [A, B]$.
\begin{lemma}
\label{lem:ad.0}
 Assume that for some  $N \in \N$  one has 
\begin{equation}
\label{ad.0}
\ad_A^j(K_0) \neq 0 \ , \ \ \forall 1 \leq j \leq N \ ,   \qquad \ad_A^{N+1} (K_0)  = 0 \ .
\end{equation}
Fix $r \in \N$ and choose  $\psi_0 \in \cH^r$ such that 
\begin{equation}
\label{ad.0.psi0}
\left[\ad_A^{N}(K_0)\right]^{r} \psi_0 \neq 0 \ .
\end{equation}
Then there exists  $C(r, N, \psi_0) >0$ such that the solution
$\psi(t)$ of \eqref{growth.os} with initial datum $\psi_0$ fulfills 
\begin{equation}
\label{ad.0.est}
\norm{\psi(t)}_r \geq C(r,  N, \psi_0) \, \la t \ra^{r N} \ , \quad t \gg 1 \ . 
\end{equation}
\end{lemma} 
The proof of the Lemma is postponed in Appendix \ref{app:A}.

\begin{remark}
\label{rem:ad.0}
Condition  \eqref{ad.0} 
can be replaced by  
$\ad_A^N (K_0') \neq 0$,  $\ad_A^{N+1} (K_0')  = 0$ where  $K_0'$ is any operator defining   norms equivalent to $\norm{\cdot}_r$.
\end{remark}

We conclude this section with an example of an operator $A$ with absolutely continuous spectrum which has unbounded orbits; such example will  guide us in the applications:
\begin{example}
\label{ex}
Let $H^r(\T) = \dom((1-\partial_{xx})^{r/2})$ be the classical Sobolev space on the one dimensional torus $\T$. Define $A = v(x)$ (multiplication operator), with $v(x) \in C^\infty(\T, \R)$ and $\grad v \not\equiv 0$. Then the equation 
 $$
\im \dot \vf = v(x) \vf \ , \qquad  \mbox{with }\vf(0) \in H^r(\T) \ \  \mbox{ and } \ \ (\grad v)\cdot \vf(0) \not\equiv 0
$$
has a solution $\vf(t) \in H^r(\T)$ fulfilling \eqref{growth.os} with $\mu = 1$.  
This follows applying Lemma \ref{lem:ad.0} and Remark \ref{rem:ad.0} with  $K_0' = \partial_x$ and 
noting  that 
$[ v(x), \partial_x] \neq 0$, $[v(x), [v(x), \partial_x]] = 0$.
\end{example}

\section{Applications}\label{sec:app}
In this section we apply  Theorem \ref{thm:main} to three different time dependent linear Schr\"odinger equations. 
In each case we construct a periodic in time perturbation which induces growth of norms.

In this section, for $\Omega\subset \R^d$ and $\cF$ a Fr\'echet space, 
we will denote by $C_b^m(\Omega, \cF)$ the space of
$C^m$ maps $f: \Omega\ni x\mapsto f(x)\in\cF$, such that,  for every
seminorm $\norm{\cdot}_j$ of $\cF$ one has
       \begin{equation}
       \label{star}
       \sup_{x\in\Omega}\Vert\partial_x^\alpha f(x)\Vert_{j} < +\infty
       \ , \quad \forall \alpha\in \N^d\  :\ \left|\alpha\right|\leq m  \ .
       \end{equation}
If \eqref{star} is true
$\forall m$, we say $f \in C^\infty_b(\Omega, \cF)$.\\

\subsection{Harmonic oscillator on $\R$}
Consider the Schr\"odinger equation
\begin{equation}
\label{ho}
\im \dot \psi = \frac{1}{2} \left(-\partial_{x}^2 + x^2\right) \psi + V(t, x, D_x) \psi \ , \qquad x \in \R \ . 
\end{equation}
Here $K_0 := \frac{1}{2} \left(-\partial_{x}^2 + x^2\right)$ is the Harmonic oscillator, the scale of  Hilbert spaces is defined  as 
usual by  $\cH^r = \dom\left(K_0^r\right)$, and the base space $(\cH^0, \la \cdot, \cdot, \ra)$ is $L^2(\R, \C)$   with its  standard  scalar product.\\
The perturbation  $V$ is chosen as the Weyl quantization of a symbol belonging to the following class
\begin{definition}
\label{symbol.ao2}
A function $f$ is a  symbol of order  $\rho \in\R$ if  $f \in C^\infty(\R_x \times \R_\xi, \R)$ and 
               $\forall \alpha, \beta \in \N$, there exists $C_{\alpha, \beta} >0$ such that
$$
 \vert \partial_x^\alpha \, \partial_\xi^\beta f( x,\xi)\vert \leq C_{\alpha,\beta} \ (1 + |x|^2 + |\xi|^2)^{\rho-\frac{|\beta| + |\alpha|}{2}}  \ . 
$$
             We will write $f \in S^\rho_{\hos}$.
\end{definition}
   We endow $S^\rho_\hos$  with the family of seminorms
$$
\wp^\rho_j(f) := \sum_{|\alpha| + |\beta| \leq j}
\ \ \sup_{(x, \xi) \in \R^{2}} \frac{\left|\partial_x^\alpha \, \partial_\xi^\beta f(x,\xi)\right|}{ \left[1 + |x|^2 + |\xi|^2\right]^{\rho-\frac{\beta +\alpha}{2}} } \ , \qquad j \in \N \cup \{ 0 \} \ . 
$$
Such seminorms turn $S^\rho_\hos$ into a Fr\'echet space.\\
If a symbol $f$ depends on additional parameters (e.g. it is time dependent), we ask that all the seminorms  are uniform w.r.t. such parameters.

 To a symbol $f \in S^\rho_\hos$ we associate the operator
$f(x, D_x)$  by standard Weyl quantization
 
       $$
   \Big( f(x, D_x) \psi \Big)(x) := \frac{1}{2\pi} \iint_{y, \xi \in \R} {\rm e}^{\im (x-y)\xi} \, f\left( \frac{x+y}{2}, \xi \right) \, \psi(y) \, \di y \di \xi \ . 
        $$


\begin{definition}
\label{pseudo.an}
  We say that $F\in \cA_\rho$ if it is a pseudodifferential operator
  with symbol of class $S^\rho_{\hos}$, i.e., if there
  exists $f \in S^\rho_{\hos}$  such that $F = f(x, D_x)$.
\end{definition}

\begin{remark} The harmonic oscillator $K_0$ has symbol given by $x^2 + \xi^2$; by our definition  $K_0 \in \cA_1$.
\end{remark}

As usual we give  $\cA_\rho$ a Fr\'echet structure by endowing it with the seminorms of the symbols.

Our first application is to construct a time dependent pseudodifferential operator of order 0 which provokes growth of Sobolev norms.
In such a way we obtain an alternative, simpler construction of the result of Delort \cite{del}.

\begin{theorem}
\label{thm:hos}
There exists a time periodic pseudodifferential operator of order 0,  $V\in C^\infty_b(\T, \cA_0)$, and for any $r \in \N$ a constant $C_r >0$ and a smooth solution $\psi(t)$ of \eqref{ho} s.t. for any $t \geq 0$, $\psi(t) \in \cH^r$ and satisfies 
$ \norm{\psi(t)}_r \geq C_r \, t^{r} $
 for $t$ large enough.
\end{theorem}
\begin{proof}
We verify that Assumptions A, B, and C are met.\\
{\sc Assumption A:} It follows from   $\sigma(K_0) = \{ n + \frac{1}{2}\}_{n \geq 0}$.\\
{\sc Assumption B:}
We define $A$ by its action on the Hermite functions $(
\be_n)_{n \in \N_0}$ (which  are eigenvectors of the Harmonic oscillator and  form a basis of $L^2(\R))$.  
We take $\delta \neq 0$ arbitrary and define
\begin{equation}
\label{def:A}
\begin{aligned}
&A\be_0 :=  \delta \be_1 \ , \qquad A\be_n := \delta  \left( \be_{n+1} + \be_{n-1}\right) \  \quad  \mbox{for }n \geq 1 \ ;
\end{aligned}
\end{equation}
the action of $A$ is extended on all $\cH^r$ by linearity, giving
$A \psi = \delta \sum_{n \geq 0} (\psi_{n-1} + \psi_{n+1}) \be_n$, where we defined  $\psi_n = \la \psi, \be_n \ra$ for $n \geq 0$ and $\psi_{-1} =0$. Clearly $A \in \cL(\cH^r)$ $\forall r \geq 0$.\\
 By Lemma \ref{lem:A.gr} below its propagator $e^{- \im t A}$  has unbounded orbits fulfilling \eqref{growth.os} with $\mu = 1$ and any initial datum.\\
{\sc Assumption C:} 
By Lemma \ref{lem:A.A0} below,  $A$ is a pseudodifferential operator of order 0,  $A \in \cA_0$. 
By   Egorov theorem for the Harmonic oscillator \cite{hero, hero2} (and using the periodicity of the flow of $K_0$)   $t \mapsto  e^{-\im t K_0} A e^{ \im t K_0}\equiv V_A(t) \in C^\infty_b(\T, \cA_0)$.
This can be seen e.g. by remarking that the symbol of $V_A(t)$ is  
$a\circ \phi_{ho}^t$, where $a \in S^0_\hos$ is the symbol of $A$ and  $\phi^t_{ho}$ is the time $t$ flow of the harmonic oscillator; explicitly 
\begin{align}
\label{symb}
\left(a\circ \phi^t_{ho}\right)(x, \xi) = a(x \cos t + \xi \sin t, -x \sin t + \xi \cos t) \ .
\end{align}
\end{proof}

\begin{remark}
The parameter $\delta$ in \eqref{def:A} can be arbitrary small; therefore, also the perturbation $V_A(t)$ can be arbitrary small in size.
\end{remark}

\begin{remark}
Consider \eqref{ho} with the perturbation   $V_A(\omega t, x, D_x)$, $\omega \in \R$, which is periodic in time with frequency $\omega$.
Then it is proved in \cite{Bam16I} that for any $\delta$  sufficiently small, there exists a Cantor set $\cO_\delta \subset [0, 1]$ such that if $\omega \in \cO_\delta$,   each solution of \eqref{ho} fulfills \eqref{up3}.
It is clear therefore that the growth of Sobolev norms depends on  resonance properties of the frequency $\omega$.
\end{remark}

\begin{remark}
In the basis of Hermite functions, the operator $A$ is the discrete Laplacian on the half line $\ell^2(\N_0)$ with Dirichlet boundary conditions, hence  it has absolutely continuous spectrum.
\end{remark}

\begin{remark}
\label{rem:os.cos}
The  Fourier transform $ (\psi_n)_{n \in \N_0} \mapsto \sum_{n \geq 0} \psi_n \sin\left((n+1)y\right) $ maps $\cH^r$ in $H^r(\T)$ and conjugates $A$ to  the multiplication operator by $\delta \cos(x)$; therefore we  are in the framework of Example \ref{ex}.
\end{remark}

\begin{remark}
One has $V_A \in C^\infty(\T, \cL(\cH^r))$ for any $r \geq 0$. 
Indeed by  Calderon-Vaillancourt theorem for the class $\cA_\rho$ (see e.g. \cite{robook}), for any $r \in \N$, there exists $N \in \N$ and $C_{r, N}>0$ such that  
$$
\sup_{t \in \T} \norm{ \partial_t^\ell  V_A(t) }_{\cL(\cH^r)} \leq C_{r,N} \,  \wp^\rho_N \left(\partial_t^\ell a(t, x, \xi)\right)  < \infty \ , \quad \forall \ell \in \N_0 \ ,
$$
where $a(t, x, \xi) \equiv (a \circ \phi^t_{ho})(x, \xi)$  is the symbol of $V_A(t)$, see \eqref{symb}, and belongs to   $C^\infty_b(\R, S^0_\hos)$.
\end{remark}

First we prove that the Hamiltonian $A$ generates unbounded orbits:
\begin{lemma}
\label{lem:A.gr}
Let $A$ be defined  in \eqref{def:A} and consider the equation 
$\im \dot \psi = A \psi$. For any $r \in \N$, any nonzero $\psi_0 \in \cH^r$, there exists a constant $C_{r, \psi_0} >0$ such that
$\norm{e^{-\im t A }\psi_0}_r \geq C_{r, \psi_0} \la t \ra^r$ for $t \gg 1$.
\end{lemma}
\begin{proof}
Remark \ref{rem:os.cos} gives essentially a proof of the statement. Alternatively, one can also apply Lemma \ref{lem:ad.0} as follows. 
Let $\be_n$ be the $n^{th}$ Hermit function;   then 
a direct computation using \eqref{def:A} and $K_0 \be_n = (n+\frac12) \be_n$  shows  that
$$
[A, K_0]\be_n = \delta (\be_{n-1} - \be_{n+1})  \ , \qquad [A, [A, K_0]]\be_n = 0 \ , \qquad \forall n \in \N_0  \ . 
$$
Then we  apply   Lemma \ref{lem:ad.0} with $N = 1$ (\eqref{ad.0.psi0} is fulfilled for any nontrivial  $\psi_0 \in \cH^r$). 
%
\end{proof}

Finally we show that $A$ is a pseudodifferential operator; to this aim,  it is necessary  to  read the pseudodifferential properties of an operator by its matrix coefficients on the basis of Hermit functions. 
Any  linear self-adjoint operator $A :\cH^\infty \to \cH^{-\infty}$ is completely determined by  its {\em matrix}  
\begin{equation}
\label{matrix}
M^{(A)} \colon \N_0 \times \N_0 \to \R  \ , \qquad
(m,n) \to \la A \be_m, \be_n \ra \ .
\end{equation}
 Define  the discrete difference operator $\triangle$ on a function $M \colon \N_0 \times \N_0 \to \R$ by
$$
(\triangle M)(m,n) := M(m+1, n+1) - M(m,n)  \ ,
$$
and its powers $\triangle^\gamma$ by  $\triangle$ applied $\gamma$-times.
Now we have the following 
\begin{definition}
A symmetric function $M \colon \N_0 \times \N_0 \to \R$ will be said to be a {\em symbol matrix of order $\rho$} if for any $\gamma, N \in \N$, there exists $C_{\gamma, N} >0$ such that 
$$
 \abs{(\triangle^\gamma M)(m,n)} \leq C_{\gamma, N} \frac{(1 + m + n)^{\rho-|\gamma|}}{(1 + |m-n|)^N}  \ .
$$
\end{definition}
The connection between pseudodifferential operators of order $\rho$ and symbol matrices of order $\rho$ is given by Chodosh's characterization \cite{chodosh}:
\begin{theorem}[Chodosh's characterization]
\label{thm:chod}
A selfadjoint operator $A$ belongs to $\cA_\rho$ if and only if its matrix  $M^{(A)}$ (as defined in \eqref{matrix})   is a symbol matrix of order $\rho$.
\end{theorem}

We can now prove: 
\begin{lemma}
\label{lem:A.A0}
The operator $A$ defined in \eqref{def:A} belongs to $\cA_0$.
\end{lemma}
\begin{proof}
By  formula \eqref{def:A}, the matrix of $A$ is given by
$$
M^{(A)}(m,n) := \la A \be_m , \be_n \ra = \delta_{n+1, m} + \delta_{n-1,m}  \ , \qquad n, m \in \N_0 \ .
$$
It is a trivial computation to verify that $M^{(A)}$ is a symbol matrix of order $0$, hence by Theorem \ref{thm:chod} it is a pseudodifferential operator in $\cA_0$. 
\end{proof}

 \subsection{Half-wave equation  on $\T$}
The half-wave equation on $\T$ is   given by 
 \begin{equation}
 \label{halfwave}
 \im \dot \psi = |D_x| \psi + V(t,x, D_x) \psi \ , \qquad x \in \T \ .
 \end{equation}
 Here  $|D_x| \psi$ is the Fourier multiplier defined by 
$$
  |D_x| \psi := \sum_{j \in \Z  } |j|\,  \psi_j \, e^{\im j x}  \ ,
 $$
where  $\psi_j := \int_\T \psi(x) e^{-\im j x} \di x$ is  the $j^{th}$ Fourier coefficient.

 We recall  now the usual    class of pseudodifferential operators on $\T$.  
For a function $f:\T\times \Z \to \R$, define the difference operator $\triangle f(x,j) := f(x,j+1)-f(x,j)$. 
 Then we have
 \begin{definition}
 A function $f: \T \times \Z \to \R$,  will be called a symbol of order $\rho \in \R$ if $ x\mapsto f(x, j)$ is smooth for any $j \in \Z$ and $\forall \alpha, \beta \in \N$, there exists $C_{\alpha, \beta}>0$ s.t.
 $$
 \abs{\partial_x^\alpha \triangle^\beta f(x, j)} \leq C_{\alpha, \beta} \, \la j \ra^{\rho - \beta}  \ .
 $$
 We will write $f \in S^\rho_\tor$.
 \end{definition}
   Again we endow $S^\rho_\tor$  with the family of seminorms
$$
\wp^\rho_j(f) := \sum_{\alpha + \beta \leq j}
\ \ \sup_{(x, j) \in \T \times  \Z} \la j \ra^{-\rho + \beta}
\left|\partial_x^\alpha \, \triangle_\xi^\beta f(x,j)\right| \ , \qquad j \in \N_0 \ . 
$$
If a symbol $f$ depends on additional parameters (e.g. it is time dependent), we ask that the constants $C_{\alpha, \beta}$ are uniform w.r.t. such parameters.

 To a symbol $f \in S^\rho_\tor$ we associate the operator $f(x, D_x)$ by standard quantization:
 $$
\Big( f(x, D_x) \psi\Big)(x) := \sum_{j \in \Z} f(x, j)\,  \psi_j \, e^{\im j x} \ .
 $$
 Then we have the following
 \begin{definition}
\label{pseudo.to}
  We say that $F\in \cA_\rho$ if it is a pseudodifferential operator
  with symbol of class $S^\rho_{\tor}$, i.e., if there
  exist $f \in S^\rho_{\tor}$ such that $F = f(x, D_x)$.
\end{definition}

\begin{remark} The operator $|D_x|$ has symbol given by $|j|$; therefore $|D_x| \in \cA_1$.
\end{remark}

As usual we give  $\cA_\rho$ a Fr\'echet structure by endowing it with the seminorms of the symbols.

In this case we have:
\begin{theorem}
Consider the half-wave equation \eqref{halfwave}. There exist  a pseudodifferential operator of order $0$, $V \in  C^\infty(\T, \cA_0)$, and for any $r \in \N$ a constant $C_r >0$ and a  solution $\psi(t)$ of \eqref{halfwave} such that, for any $t \geq 0$,  $\psi(t) \in H^r(\T)$ and satisfies $\norm{\psi(t)}_{H^r(\T)} \geq C_r \, t^{r}$ for $t$ large enough.
\end{theorem}
\begin{proof}
First we show how to put ourselves in the setting of  the abstract problem. Define 
 $$
 K_0 := |D_x| + \lambda  \ , \qquad \lambda >0 \ . 
 $$
 The space $\cH^r :=\dom(K_0^r)$ coincides with  $H^r(\T)$, with equivalent norms.
 We take  perturbations of the form  $V(t, x, D_x) = \lambda\left( 1+ e^{-\im t K_0} A e^{\im t K_0}\right)$, so that  \eqref{halfwave}  becomes
 $$
 \im \dot \psi = K_0 \psi  + V_A(t, x, D_x) \psi \ , \qquad V_A(t, x, D_x) := \lambda\, e^{-\im t K_0} A e^{\im t K_0} \ , 
 $$
and we are back to the abstract setting. 
We need only to verity Assumptions A, B, C.
 
 {\sc Assumption A:} Trivial, since $\sigma(K_0) = \{ j +\lambda \}_{j \in \N_0}$.
 
 {\sc Assumption B:} Take $A$ and the initial datum $\psi_0$   as in Example \ref{ex}. 
 
 {\sc Assumption C:} One has  $A \in \cA_0$ and $K_0 \in \cA_1$. By a classical result of H\"ormander \cite{ho} (see also  \cite{dui}),   $t \mapsto e^{-\im t K_0} A e^{\im t K_0} \in C^\infty(\R, \cA_0)$; actually, being periodic in time, it belongs to $C^\infty(\T, \cA_0)$.
\end{proof}

\begin{remark}
The parameter $\lambda$ can be arbitrary small; therefore also in this case  $V(t,x, D_x)$ can be arbitrary small in size.
\end{remark}

\begin{remark}
By Calderon-Vaillancourt theorem,   $V(t,x, D_x) \in C^\infty(\T, \cL(\cH^r))$  $\, \forall r \in \N$.
\end{remark}

\subsection{Schr\"odinger-Dirac equation on Zoll manifolds}
Consider the Schr\"odinger-Dirac equation on a Zoll manifold $M$ (e.g., $M$ can be a $n$-dimensional sphere)
\begin{equation}
\label{zoll}
\im \dot \psi = \sqrt{-\Delta_{g} + m^2 } \,  \psi + V(t,x, D_x) \psi \ , \qquad x \in M  \ ;
\end{equation}
here $m \neq 0$ is a real number  and $-\Delta_{g}$ is the positive  Laplace-Beltrami operator on  $M$. 
Let  $H^r(M) = \dom\left((1-\Delta_g)^{r/2}\right)$, $r\geq 0$, the usual scale of Sobolev spaces on $M$.  Finally we denote by
      $S_{\rm cl}^\rho$ the space of classical real valued symbols of
      order $\rho\in \R$ on the cotangent $T^*(M)$ of $M$ (see
      H\"ormander \cite{ho} for more details).
      \begin{definition}
        \label{pseudo}
We say that $F\in\cA_\rho$ if it is a pseudodifferential operator (in the sense of H\"ormander \cite{ho}) with
symbol of class $S^\rho_{\rm cl}$.
      \end{definition}

\begin{remark} The operator $\sqrt{-\Delta_{g} + m^2 }$ belongs to  $ \cA_1$ \cite{ho}.
\end{remark}      
      
\begin{remark}
By \cite{cdv}, there exist $ c_0, c_1 >0$ such that 
$$
\sigma\left(\sqrt{-\Delta_{g} + m^2 } \right) \subseteq \bigcup_{j \geq 0} \left[ j + c_0 - \frac{c_1}{j} , \ j + c_0 + \frac{c_1}{j} \right] \ , 
$$
 so in this case  the spectral gaps are only {\em asymptotically} constant. 
\end{remark}      
      
      We have the following 
     \begin{theorem}
      Consider the Schr\"odinger-Dirac  equation \eqref{zoll}. There exists a pseudodifferential operator of order $0$, $V \in  C^\infty(\T, \cA_0)$, and for any $r \in \N $ a constant $C_r >0$ and a  solution $\psi(t)$ of \eqref{zoll} fulfilling $\psi(t) \in H^r(M)$ for any time $t$, and 
\begin{equation*}
 \norm{\psi(t)}_{H^r(M)} \geq C_r \, t^{r} \ , \qquad t \gg 1 \ .
 \end{equation*}
\end{theorem}
\begin{proof}
    To begin with we show how to put ourselves in the abstract setup. So first we  define the operator $K_0$. This is achieved by exploiting the spectral properties of the operator $-\Delta_g$.     
     Applying  Theorem 1  of Colin de Verdi\`ere \cite{cdv},  there  exists a pseudodifferential operator $Q$  of order $-1$, commuting with $-\Delta_g$,
       such that ${\rm Spec}[\sqrt{-\Delta_g + m^2} + Q]\subseteq \N+\lambda$  with some $\lambda>0$.
        So we define 
\begin{equation}
\label{zoll.k0}
K_0 := \sqrt{-\Delta_g + m^2} + Q  \in \cA_1 \ .
\end{equation}   
Since $Q \in \cA_{-1}$, the space  $\cH^r := {\rm Dom } (K_0^r)$, $r \geq 0$,  coincides with the classical Sobolev space $H^r(M)$ and one has the equivalence of norms
$$
c_r \, \norm{\psi}_{H^r(M)} \leq \norm{\psi}_r \leq \wt c_r \, \norm{\psi}_{H^r(M)} \ , \qquad \forall r \in \R \ . 
$$
We take the perturbation of the form 
$V(t,x, D_x) = Q + e^{-\im t K_0} A e^{\im t K_0}$, 
so that \eqref{zoll} becomes  
$$
\im \dot \psi = K_0 \psi + V_A(t, x, D_x) \ , \qquad V_A(t, x, D_x) =  e^{-\im t K_0} A e^{\im t K_0}
$$
and we are back to the abstract setting.
We need only to verify Assumptions A, B, C.

{\sc Assumption A:} True by construction.  

{\sc Assumption B:}  It follows by a trivial variant of Example \ref{ex}.
Choose any   nonconstant $v(x)\in C^\infty(M, \R)$, define $A$ as the multiplication operator by $v(x)$, 
and take an initial datum $\psi_0 \in \cH^r$ fulfilling $(\nabla_g v(x)) \psi_0 \not\equiv 0$. 
For example,  $v(x)$ can be  any  noncostant  eigenfunction  of $-\Delta_g$.
Then the Schr\"odinger equation \eqref{eq:A} has orbits fulfilling \eqref{growth.os} with $\mu = 1$
(it is enough to apply Lemma \ref{lem:ad.0} and  Remark \ref{rem:ad.0} using   $[\nabla_g, v(x)] \neq 0$, $[[\nabla_g, v(x)], v(x)] = 0$).

{\sc Assumption C:} One has  $A \in \cA_0$ and $K_0 \in \cA_1$. Then  $e^{-\im t K_0} A e^{\im t K_0} \in C^\infty(\R, \cA_0)$ by a classical result of H\"ormander \cite{ho}.
\end{proof}
 
\begin{remark}
In this case, the perturbation $V(t, x, D_x)$ cannot be chosen arbitrary small in size, since we have to  add the smoothing operator $Q$ to 
correct the spectral gaps.
\end{remark} 
 
\begin{remark}
One could also choose $V(t, x, D_x)$ as $ e^{-\im t \sqrt{- \Delta_g + m^2}} \, \epsilon   A \, e^{\im t \sqrt{-\Delta_g + m^2}}$; in such a way the perturbation is arbitrary small in size and  fulfills Assumption {\rm C} (again by \cite{ho}), but it is not periodic in time.
\end{remark}

 \appendix
 
 \section{Proof of Lemma \ref{lem:ad.0}}
 \label{app:A}

 Since the linear operator $\ad_A$ fulfills Leibniz rule,   for any $M, r \in \N$ one has the identity
 \begin{equation}
 \label{ad.ind}
 \ad_A^M(K_0^{2r}) = \sum_{k_1 + \ldots + k_{2r} = M} \binom{M}{k_1  \cdots k_{2r}} \, \ad_A^{k_1}(K_0) \, \ad_A^{k_2}(K_0)  \cdots \ad_A^{k_{2r}}(K_0) 
 \end{equation}
If $M \geq 2rN +1$ then in 
 \eqref{ad.ind}  at least one index $k_j$ is greater equal $N+1$, so by assumption \eqref{ad.0} the whole expression is zero.
By the same argument,  if $M = 2rN$ then the only term not null in \eqref{ad.ind} is the one with $k_j = N$ $\, \forall j$, which is  
$\left[ \ad_A^{N}(K_0) \right]^{2r}$.

Consider now the solution  $\psi(t)\equiv e^{- \im t A} \psi_0$  of equation \eqref{eq:A}.  
Since $A$ is  self-adjoint,  
 $$
 \norm{\psi(t)}_r^2 \equiv \la e^{\im t A} \, K_0^{2r} \, e^{- \im t A} \psi_0, \psi_0 \ra \ ,
 $$
where   we used $\norm{\psi}_r^2 \equiv \la K_0^{2r} \psi, \psi \ra$.
Now we use the Lie formula $ e^{\im t A} \, B \, e^{- \im t A}\equiv \sum_{j \geq 0} \frac{(\im t)^j}{j!} \ad_A^j (B) $, assumption \eqref{ad.0} and our previous considerations to obtain
 $$
 e^{\im t A} \, K_0^{2r} \, e^{- \im t A} = \sum_{M=0}^{2rN}
 \frac{(\im t)^M}{M !} \, \ad_{A}^M(K_0^{2r}) \equiv \frac{(\im t)^{2rN}}{(2rN)!} \left[\ad_A^{N}(K_0) \right]^{2r} + O(t^{2rN-1}) \ ;
 $$
provided $\left[\ad_A^{N}(K_0)\right]^{r} \psi_0 \neq 0$,  it follows that 
 \begin{equation}
 \label{}
 \liminf_{t \to +\infty} \frac{\norm{\psi(t)}_r^2}{|t^{2rN}|} \geq 
 \frac{1}{(2rN)!}\ \abs{ \la \left[\ad_A^{N}(K_0) \right]^{2r} \psi_0, \psi_0 \ra} > 0  \ . 
 \end{equation}
In particular there exists a constant $C(r, N, \psi_0) >0$ such that  
 \eqref{ad.0.est} holds true.

\bibliographystyle{alpha} 
\newcommand{\etalchar}[1]{$^{#1}$}
\def\cprime{$'$}

\end{document}